\documentclass[reqno,12pt]{amsart} %amsproc/amsbook
\usepackage{amssymb,amsmath,mathrsfs,bbm,color}
\usepackage[colorlinks=true, pdfstartview=FitV, linkcolor=blue, citecolor=red, urlcolor=blue,pagebackref=false]{hyperref}
\usepackage[font=footnotesize]{caption}
\usepackage{tikz}
\usepackage{float}
\usepackage{esint}
\usepackage{bookmark}
\usepackage[left=1in, right=1in, bottom=1.5in]{geometry} %top=1.5in,
\newtheorem{theorem}{Theorem}[section]

\theoremstyle{definition}

\theoremstyle{remark}

\numberwithin{equation}{section}

\begin{document}

\title[]{Upper bound of critical sets of solutions of elliptic equations in the plane }


%    author one information
%\author {Carlos E. Kenig}
%\address{
%	Department of Mathematics\\
%	University of Chicago\\
%	Chicago, IL 60637, USA\\
%	Email:  cek@math.uchicago.edu }

\author{ Jiuyi Zhu}
\address{
Department of Mathematics\\
Louisiana State University\\
Baton Rouge, LA 70803, USA\\
Email:  zhu@math.lsu.edu }
%\author {Jinping Zhuge}
%\address{
%Department of Mathematics\\
%University of Chicago\\
%Chicago, IL 60637, USA\\
%Email:   jpzhuge@math.uchicago.edu }

\thanks{Zhu is supported in part by NSF grant OIA-1832961 and  DMS-2154506}
%\date{\today}

\dedicatory{Dedicated to Professor Carlos E. Kenig
on the Occasion of His 70th Birthday}
\begin{abstract} In this note, we investigate the measure of singular sets and critical sets of real-valued solutions of elliptic equations in two dimensions. These singular sets and critical sets are finitely many points in the plane. Adapting the Carleman estimates involving polynomial functions at singularities by Donnelly and Fefferman in \cite{DF90}, we obtain the upper bounds of singular points and critical points.
\end{abstract}

\maketitle
\section{Introduction}
We consider the upper bounds of singular sets for real-valued solutions of elliptic equations
\begin{align}
{\rm div}( A(x)\nabla u)+ b(x)\cdot \nabla u+c(x) u=0 \quad \quad \mbox{in}  \ \mathbb B_5,
\label{main-equ-s}
\end{align}
and critical sets for real-valued solutions of elliptic equations
\begin{align}
{\rm div}( A(x)\nabla u)+ b(x)\cdot\nabla u=0 \quad \quad \mbox{in}  \ \mathbb B_5,
\label{main-equ-c}
\end{align}
where $A(x)=(a_{ij}(x))_{2\times 2}$ is real-valued Lipschitz continuous, $b(x)=(b_1(x), b_2(x))$, $c(x)$ are bounded functions in the plane and $\mathbb B_5$ is the ball centered at origin with radius $5$. Especially, we
assume that $A(x)$ satisfies the uniform ellipticity conditions
\begin{align}
\Lambda_1 |\xi|^2\leq a_{ij}(x)\xi_i \xi_j \leq \Lambda_2 |\xi|^2,
\label{ellip}
\end{align}
and the Lipschitz continuity conditions
\begin{align}
|a_{ij}(x)-a_{ij}(y)|\leq \Lambda_0 |x-y| \quad \mbox{for any} \  x, y\in \mathbb B_5.
\label{lipschitz}
\end{align}
The functions $b(x)$ and $c(x)$ are bounded as
\begin{align}
\| {b}\|_{L^\infty(\mathbb B_5)} \leq M_0, \quad \|c\|_{L^\infty(\mathbb B_5)}\leq M_1.
\label{bboud}
\end{align}

 The singular sets are given by $\mathcal{S}=\{ x\in \mathbb B_2| u(x)=|\nabla u(x)|=0\}$. The critical sets are defined as $\mathcal{C}=\{ x\in \mathbb B_2| |\nabla u(x)|=0\}$. It is known that the singular sets and critical sets are finitely many discrete points for (\ref{main-equ-s}) and (\ref{main-equ-c}) in the plane, see e.g. the implicit bound of singular sets \cite{HHL98} and critical sets \cite{HHON99} for elliptic equations for any dimensions using compactness arguments. Around each singular point, the nodal sets consist of finitely many curves intersecting at this point with equal angles. The set of singular points is a subset of critical points. There are two types of points for critical points. One are the singular points, i.e. $\nabla u(x)=0$ and $u(x)=0$.  The other are the non-sigular points, i.e. $\nabla u(x)=0$, but $u(x)\not=0$.

  For real-valued harmonic functions, that is,  $A(x)=\delta_{ij}$, $b(x)=0$ and $c(x)=0$ in (\ref{main-equ-c}), Han \cite{H04} showed that $H^0(\mathcal{S})\leq C\mathcal{N}(4)$, where $\mathcal{N}(r)$ is the frequency function given by
  \begin{align}
  \mathcal{N}(r)=\frac{ r\int_{\mathbb B_r} |\nabla u|^2} {\int_{\partial \mathbb B_r}  u^2}.
  \label{frequency}
  \end{align}
Such upper bound can be also obtained by the analyticity of harmonic functions in \cite{L91}.
For complexification of real-valued function $u$, say $\tilde{u}$, the upper bound of singular sets $H^0(\{z\in \mathbb D_1| \tilde{u}(z)= \tilde{u}_{z_1} (z)=\tilde{u}_{z_2} (z)=0 \})\leq C\mathcal{N}^2(4)$ is obtained in \cite{H04}. Especially, some example is constructed in \cite{H04} to indicate that the real-valued property of solution is necessary to have the upper bound $H^0(\mathcal{S})\leq C\mathcal{N}(4)$. For the upper bound of singular sets of solutions in (\ref{main-equ-s}) for any dimension $n\geq 3$, an important conjecture
\begin{align} H^{n-2}(\mathcal{S})\leq C\mathcal{N}^2(4) \end{align}
was raised by Lin in \cite{L91}.
Naber and Valtorta \cite{NV17} obtained an exponential  upper bound of volume estimates for effective singular sets of (\ref{main-equ-s}) and effective critical sets of (\ref{main-equ-c}) using the new arguments of almost cone splitting  in \cite{CNV15} and the covering lemma in any dimensions. Especially, the exponential upper bound holds for singular points
$H^0(\mathcal{S})\leq e^{ C\mathcal{N}(4)}$ and for critical points $H^0(\mathcal{C})\leq e^{C\mathcal{\hat{N}}(4)}$ in \cite{NV17} in the plane, where \begin{align}\mathcal{\hat{N}}(r)=\frac{ r\int_{\mathbb B_r} |\nabla u|^2} {\int_{\partial \mathbb B_r} ( u-u(0))^2} \end{align} is the modified version of the frequency function (\ref{frequency}) for the study of critical sets. Note that we have considered the rescaled version in the aforementioned results.
It is also interesting to study the bounds of singular sets and critical sets of eigenfunctions. For singular sets, see \cite{DF90}, \cite{D92} for the upper bound of singular sets of Laplace eigenfunctions on surfaces,  and \cite{Z16} for the upper bound of singular sets of Steklov eigenfunctions on surfaces. For critical sets, see \cite{JN99} for bounded  number of critical points  and \cite{BLS21} for unbounded number of critical points of Laplace eigenfunctions with  some given Riemannian metrics on two dimensional torus $\mathbb T^2$, and \cite{Z21} for discussions of the upper bound of critical sets for Dirichlet eigenfunctions.

%  We can construct geodesic coordinates, we can reduce to the Laplace-Beltrami operator on surfaces, then by the use of isothermal coordinate.
%\begin{align}
%\triangle u+ \hat{b}\cdot \nabla u=0.
%\label{main-redu}
%\end{align}

Let us introduce the double index for $u$ as
\begin{align}
N(u, r)=\log_2 \frac{\|u\|_{L^\infty(\mathbb B_{2r})}} {\|u\|_{L^\infty(\mathbb B_{r})}}.
\label{double-u}
\end{align}
Frequency function $\mathcal{N}( r)$ in (\ref{frequency}) characterizes the growth rate of the solutions. It implies that the bounds of double index $N(u, r)$. It is well-known that the frequency function $\mathcal{N}( r)$ is almost monotone, i.e. $e^{Cr}\mathcal{N}( r)$ is monotone for $0\leq r\leq r_0$, where $C$ and $r_0$ depend on the coefficients in (\ref{main-equ-s}). Based on the monotonicity of the frequency function,  $\mathcal{N}(r)$ and $N(u, r)$ are comparable in the sense that
\begin{align}
C_1 \mathcal{N}( r)-C\leq N(u, r)\leq C_2 \mathcal{N}( 3r)+C
\label{dou-com}
\end{align}
for $0<r\leq R_1\leq \frac{r_0}{3}$, where $0<C_1<1$, $C_2>1$, $C$ and $R_1$ depend on the coefficients in (\ref{main-equ-s}). Furthermore, we can get that the almost monotonicity of double index,
\begin{align}
 N(u, r) \leq C_3 N(u, tr)
 \label{mono-sin-1}
\end{align}
for $t>2$ and $0<r<R_0\leq R_1$, where $R_0$ depends on the coefficients in (\ref{main-equ-s}). See e.g. \cite{HL}, \cite{Lo18}.
Assume that $N=N(u, \frac{R_0}{2})\geq 1 $ is large. Then it follows from (\ref{dou-com}) that $ \mathcal{N}(\frac{3R_0}{2})$ is large. Our first result is to show the following upper bound of singular points.
\begin{theorem}
Assume that $u$ satisfies the equation (\ref{main-equ-s}). Then it holds that
\begin{align}
H^{0}(\{\mathcal{S}\cap \mathbb B_{\frac{R_0}{25}}\}   )\leq C \mathcal{N}(\frac{3R_0}{2}),
\end{align}
where $R_0$ depends on $\Lambda_0$, $\Lambda_1$, $\Lambda_2$, $M_0$ and $M_1$.
\label{th1}
\end{theorem}

To study the critical points in (\ref{main-equ-c}), we introduce the double index for $\nabla u$ as
\begin{align}
\hat{N}(\nabla u, r) =\log_2 \frac{\|\nabla u\|_{L^\infty(\mathbb B_{2r})}} {\|\nabla u\|_{L^\infty(\mathbb B_{r})}}.
\label{C-double}
\end{align}

The frequency function $\hat{\mathcal{N}}( r)$ is almost monotone, i.e. $e^{Cr}\hat{\mathcal{N}}( r)$ is monotone for $0\leq r\leq r_0$, where $C$ and $r_0$ depend on the coefficients in (\ref{main-equ-c}). This monotonicity of the frequency function implies that $\hat{\mathcal{N}}(r)$ and $\hat{N}(u, r)$ are comparable in the sense that
\begin{align}
C_1 \hat{\mathcal{N}}( r)-C\leq \hat{N}(\nabla u, r)\leq C_2 \hat{\mathcal{N}}( 3r)+C
\label{dou-com-c}
\end{align}
for $0<r\leq R_1\leq \frac{r_0}{3}$, where $0<C_1<1$, $C_2>1$, $C$ and $R_1$ depend on the coefficients in (\ref{main-equ-c}). Moreover, the almost the monotonicity of double index holds,
\begin{align}
 \hat{N}(\nabla u, r) \leq C_3 \hat{N}(\nabla u, tr)
 \label{mono-cri-1}
\end{align}
for $t>2$ and $0<r<R_0\leq R_1$, where $R_0$ depends on the coefficients in (\ref{main-equ-c}). See e.g. \cite{LM18}, \cite{NV17}.
 Assume that $\hat{N}=\hat{N}(\nabla u, \frac{R_0}{2})\geq 1$ is large. Our second result is to show that
\begin{theorem}
Assume that $u$ satisfies the equation (\ref{main-equ-c}). Then it holds that
\begin{align}
H^{0}(\{\mathcal{C}\cap \mathbb B_{\frac{R_0}{25}}\}   )\leq C\hat{\mathcal{N}}(\frac{3R_0}{2}),
\end{align}
where $R_0$ depends on $\Lambda_0$, $\Lambda_1$, $\Lambda_2$ and $M_0$.
\label{th2}
\end{theorem}

%\begin{remark}
 %The length of level lines of solutions of elliptic equations in the plane
% \end{remark}

This note is organized as follows. In section 2, we reduce the second order elliptic operators with Lipschitz leading coefficients to the Euclidean Laplace operators. Then we present the Carleman estimates involving polynomial functions at singularities in the plane. Section 3 is devoted to the derivation of upper bounds of singular sets and critical sets in Theorem \ref{th1} and \ref{th2}.  The letter $C$ and $C_i$ denote some generic positive constants and do not depend on $u$. It may vary in different
lines and sections.

\noindent{\bf Acknowledgement:} The author would like to thank Professor Alexander Logunov for reading the paper and helpful discussions. After the paper was posted in arXiv, the author was informed by Professor Giovanni Alessandrini that the similar conclusion in Theorem \ref{th2} was shown by a different approach in \cite{A88}. It also was pointed out that some interesting estimates about the number and multiplicities of critical points with the boundary data were shown in e.g. \cite{A87}, \cite{AM92}.

\section{Carleman estimates for Euclidean Laplace }
In this section, we first  construct  Lipschitz metrics from the Lipschitz leading coefficient $A(x)$ in  (\ref{main-equ-s}) and (\ref{main-equ-c}). The arguments are adapted from \cite{AKS62}. We present the details for the convenience of the readers. Then we reduce the study on Laplace-Beltrami operator to Euclidean Laplace by isothermal coordinates. At last, we present the Carleman estimates involving polynomials for Euclidean Laplace.
 Without loss of generality, we consider the construction of geodesic coordinates at origin.   We introduce a ``radial" coordinate and a conformal change metric ${g}_{ij}$ in $\mathbb B_{r_0}$,
 \begin{equation}
 r=r(x)=({a}^{ij}(0)x_i x_j)^{\frac{1}{2}}
 \label{radial}
 \end{equation}
 and
\begin{equation*}
{g}_{ij}(x)= {a}^{ij}(x){\hat{\psi}}(x),
\end{equation*}
where \begin{equation} {\hat{\psi}}(x)= {a}_{kl}(x)\frac{\partial r}{\partial x^k}\frac{\partial r}{\partial x^l}\end{equation} for $x\not=0$ and  $({a}^{ij})=({a}_{ij})^{-1}$ is the inverse matrix. In the whole paper, we adopt the Einstein notation. The summation of index is understood. From the assumption of (\ref{ellip}),  $\hat{\psi}$ is bounded above and below as
\begin{equation*}
\frac{\Lambda_1}{\Lambda_2}\leq \hat{\psi}\leq \frac{\Lambda_2}{\Lambda_1}.
\end{equation*}
 It is easy to see that $\hat{\psi}$ is Lipschitz continuous. With these auxiliary quantities, the following replacement of geodesic polar coordinates are constructed in \cite{AKS62}.
In the geodesic ball
$\hat{ \mathbb B}_{\hat {r}_0}=\{ x\in  \mathbb B_{r_0}| r(x)\leq \hat{r}_0\},$ the following properties hold: \medskip \\
(i) ${g}_{ij}(x) $ is Lipschitz continuous;\medskip\\
(ii) ${g}_{ij}(x) $ is uniformly elliptic with $ \frac{\Lambda_1}{ \Lambda^2_2 }\|\xi\|^2\leq {g}_{ij}(x)\xi_i\xi_j\leq \frac{\Lambda_2}{ \Lambda^2_1}\|\xi\|^2.  $ \medskip\\
(iii) Let $\Sigma =\partial \hat{ \mathbb B}_{\hat {r}_0}$. We can parametrize $\hat{ \mathbb B}_{\hat {r}_0} \backslash \{0\}$ by the polar coordinate $r$ and $\theta$, with $r$ defined by (\ref{radial}) and $\theta$ be the local coordinates on $\Sigma$. In these polar coordinates, the metric can be written as
\begin{equation}
{g}_{ij}(x)  dx^i dx^j= dr^2+ r^2 {\gamma} d\theta d\theta
\end{equation}
with ${\gamma}=\frac{1}{r^2} {g}_{kl}(x) \frac{\partial x^k}{\partial \theta}\frac{\partial x^l}{\partial \theta}$. \medskip\\
%(iv) There exists a positive constant $M_2$ depending on ${a}_{ij}$ such that for any tangent vector $\xi_j\in T_{\theta}(\Sigma)$,
%\begin{align}
%|\frac{\partial {\gamma}_{ij}(r, \theta)}{\partial r} \xi^i \xi^j|\leq M| {\gamma}_{ij}(r, \theta)\xi^i \xi^j|.
%\label{spheress}
%\end{align}
%Let ${\gamma}=\det{ ({\gamma}_{ij})}$. Then (\ref{spheress}) implies that
%\begin{equation}
%|\frac{\partial \ln \sqrt{{\gamma}}}{\partial r}| \leq CM.
%\end{equation}

The existence of the coordinates $(r, \ \theta)$ allows us to pass to ``geodesic polar coordinates".
In particular, $r(x)=({a}^{ij}(0)x_i x_j)^{\frac{1}{2}}$ is the geodesic distance to the origin in the metric ${g}_{ij}$. Thus, we may identify $\hat{ \mathbb B}_{\hat {r}_0}$ as the Euclidean ball $\mathbb B_{\hat {r}_0}$.
 The  Laplace-Beltrami operator is given as $$\triangle_{{g}}=\frac{1}{\sqrt{{g} }} \frac{\partial}{\partial x_i}( {g}^{ij}\sqrt{{g}} \frac{\partial}{\partial x_j}   ),$$ where ${g}= \det({g}_{ij})$.  If $ u$ is a solution of (\ref{main-equ-s}), in the new metric ${g}_{ij}$, then $ u$ is locally the solution of the equation
\begin{equation}
\triangle_{{g}}  u +\hat{b}(x)\cdot\nabla_g  u+ \hat{c}(x)  u=0 \quad \mbox{in} \ {\mathbb B}_{\hat{r}_0},
\label{target}
\end{equation}
where
\begin{equation}
\left\{ \begin{array}{lll}
&\hat{b}_i= -\frac{1} {2\tilde{a}\hat{\psi}} \frac{\partial \tilde{a}}{\partial x_i} +\frac{1}{\hat{\psi}}{b}_i,   \medskip \\
&\hat{c}(x)=\frac{{c}(x)}{\hat{\psi}},
\end{array}
\right.
\end{equation}
and $\tilde{a}=\det(a^{ij})$.
By the Lipschitz continuity of $A(x)$, then $\tilde{a}$ is Lipschitz continuous. Hence $\hat{b}=(\hat{b}_1, \hat{b}_2)$ is bounded.
By the properties of $\hat{\psi}$, and the conditions (\ref{bboud}) on $b$ and $c$, we still write the conditions for $\hat{b}$ and $\hat{c}$ as
\begin{equation}
\left\{ \begin{array}{lll}
\|\hat {b}\|_{L^\infty({\mathbb B}_{\hat{r}_0})}\leq CM_0, \medskip \\
\|\hat { c}\|_{L^\infty({\mathbb B}_{\hat{r}_0})}\leq C M_1,
\end{array}
\right.
\label{targetcon}
\end{equation}
where $C$ depends on $\Lambda_1$ and $\Lambda_2$.
By the same construction of Lipschitz metric $g$, the solution $u$ in the equation (\ref{main-equ-c}) satisfies
\begin{equation}
\triangle_{{g}}  u +\hat{b}(x)\cdot\nabla_g   u=0 \quad \mbox{in} \ {\mathbb B}_{\hat{r}_0},
\label{target--}
\end{equation}
with $\hat{b}= -\frac{1} {2\tilde{a}\hat{\psi}} \frac{\partial \tilde{a}}{\partial x_i} +\frac{1}{\hat{\psi}}{b}_i  $
and $\|\hat {b}\|_{L^\infty({\mathbb B}_{\hat{r}_0})}\leq CM_0$.

Applying the isothermal coordinates for the surfaces with Lipschitz Riemannian metrics in  \cite{Chern55} or \cite{HW55} (or the so called pseudo-analyticity used on page 79 in \cite{B53}),
we have $\triangle_g =\phi(x)^{-1}\triangle$, where $\phi(x)>0$ is continuous. Therefore, we can write (\ref{target}) as
\begin{equation}
\triangle  u +\tilde {b}(x)\cdot\nabla  u+ \tilde{c}(x)  u=0 \quad \mbox{in} \ {\mathbb B}_{\hat{r}_0}
\label{target-1}
\end{equation}
and (\ref{target--}) as
\begin{equation}
 \triangle u +\tilde {b}(x)\cdot\nabla   u=0 \quad \mbox{in} \ {\mathbb B}_{\hat{r}_0},
\label{target-2}
\end{equation}
where $\tilde {b}(x)$ and $\tilde{c}(x)$ satisfy the same conditions as (\ref{targetcon}).

Next we will establish Carleman estimates involving polynomial functions at singularities for differential operators with the Euclidean Laplace as the leading term. It is directly from the Carleman estimates from \cite{DF90}. We present the details to show the role of real-valued functions and the double index $N$ in the Carleman estimates.
Let
\begin{align*} \overline{\partial}=\frac{1}{2}(\frac{\partial}{\partial x_1}-i \frac{\partial}{\partial x_2}) \quad \mbox{and} \quad  {\partial}=\frac{1}{2}(\frac{\partial}{\partial x_1}+i \frac{\partial}{\partial x_2}).
 \end{align*} Note that $\overline{\partial}{\partial}=\frac{1}{4}\triangle$. Let $P(z)= \prod (z-z_i)^{d_i}$ for some $d_i\geq 0$, where $z=x_1+ i x_2$ and $z_i=x_1^i+ ix_2^i$ are in the complex plane.  It is shown in \cite{DF90} that
\begin{align}
\int_{\mathbb D_5}| \overline{\partial} F|^2 |P|^{-2}e^{C \alpha |z|^2} \geq
 C\alpha \int_{\mathbb D_5}| F|^2|P|^{-2} e^{C \alpha |z|^2}
\label{complex-car}
\end{align}
for any smooth (possibly complex-valued) function $F\in C^\infty_0(  \mathbb D_5\backslash \cup \mathbb D_i(z_i))$ and positive constant $\alpha$. Here $\mathbb D_5$ is a ball in the complex plane with radius $5$ and $\mathbb D_i(z_i)$ are some small pairwise disjoint balls centered at $z_i$ with radius $\delta$.
Let $f\in C^\infty_0(\mathbb D_5\backslash \cup \mathbb D_i(z_i))$ be a real-valued function. We will show the following Carleman estimates hold
\begin{align}
\int_{\mathbb B_5}|\triangle f+ \tilde{b}(x)\cdot \nabla f + \tilde{c}(x) f|^2 |P|^{-2} e^{C N|z|^2}\geq CN^2 \int_{\mathbb B_5}| f|^2|P|^{-2} e^{C N|z|^2}.
\label{main-Car}
\end{align}

 Choosing $F=\partial f$ and $\alpha=N$ in (\ref{complex-car}), we obtain
\begin{align}
\int_{\mathbb D_5}| \triangle f|^2 |P|^{-2} e^{C N|z|^2}\geq
 CN \int_{\mathbb D_5}| \partial f|^2|P|^{-2} e^{C N|z|^2}.
 \label{NNN}
\end{align}
Since $f$ is a real-valued function, it holds that
\begin{align*}
|\overline{\partial} f|=|\partial f|=\frac{1}{2}|\nabla f|.
\end{align*}
Let $F=f$ in (\ref{complex-car}) and $\alpha=N$. We have
\begin{align}
\int_{\mathbb D_5}| \nabla f|^2 |P|^{-2}e^{C N|z|^2} \geq
 CN \int_{\mathbb D_5}| f|^2|P|^{-2} e^{C N|z|^2}.
 \label{gra-c}
\end{align}
Furthermore,
it follows from (\ref{complex-car}) and (\ref{NNN}) that
\begin{align}
\int_{\mathbb D_5}| \triangle f|^2 |P|^{-2}e^{C N|z|^2} \geq
 CN^2 \int_{\mathbb D_5}| f|^2|P|^{-2} e^{C N|z|^2}.
 \label{lap-c}
\end{align}
In order to consider the equation (\ref{main-equ-s}), we need to take the first order term and zero order term into considerations. Assume that $N\geq CM_0$ and $N\geq CM_1$. We identify $\mathbb D_5$ as $\mathbb B_5$, and $z_i=(x^i_1, x^i_2)$ in $\mathbb R^2$. Thus, the inequalities (\ref{gra-c}) and (\ref{lap-c}) hold for $f\in C^\infty_0(\mathbb B_5\backslash \cup D_i(z_i))$, where $ D_i(z_i)$ are  small pairwise disjoint balls centered at $z_i$ with radius $\delta$ in $\mathbb B_5$.  By the triangle inequality, (\ref{targetcon}), (\ref{gra-c}) and (\ref{lap-c}), we can obtain the following Carleman estimates,
\begin{align}
\int_{\mathbb B_5}| \triangle f+ \tilde{b}(x)\cdot \nabla f+ \tilde{c}(x)f|^2 |P|^{-2}e^{C N|z|^2} &\geq
\int_{\mathbb B_5}  | \triangle f|^2   |P|^{-2}e^{C N|z|^2}- CM_1\int_{\mathbb B_5}  | \nabla f|^2   |P|^{-2}e^{C N|z|^2} \nonumber \\
&-CM_0 \int_{\mathbb B_5}  |f|^2   |P|^{-2}e^{C N|z|^2}\nonumber \\
&\geq CN^2 \int_{\mathbb B_5}  |f|^2   |P|^{-2}e^{C N|z|^2}\nonumber \\ &-CM_1 N \int_{\mathbb B_5}  | f|^2   |P|^{-2}e^{C N|z|^2} \nonumber \\
&\geq CN^2 \int_{\mathbb B_5}  |f|^2   |P|^{-2}e^{C N|z|^2}.
\end{align}
Thus, the Carleman estimates (\ref{main-Car}) are arrived.

\section{Upper bounds of Singular points and Critical points}

In this section, we first study the upper bound of  singular points for (\ref{main-equ-s}) by adapting the arguments in \cite{DF90}. We may choose $R_0\leq \frac{\hat{r}_0}{2}$. By rescaling, we set $R_0=5$ and $\hat{r}_0=10$. Suppose that $|z_i|\leq \frac{1}{5}$. We first consider the singular points with the vanishing order more than two. Assume that $u$ vanishes at $z_i$ with order $n_i+1$, where $n_i=d_i+1$ with $d_i\geq 1$. If the vanishing order is two, i.e. $n_i=1$, we will consider it with some special argument later on.
Near the singular point $z_i$, $u(x)$ can be approximated by a homogeneous polynomial with degree $d_i+2$ in $\hat{D}_i(z_i)$, where $\hat{D}_i$ are the pairwise disjoint small  disks with radius $\delta_i$ centered at $z_i$, see e.g. \cite{B55}. Since singular points are discrete and finite, such small $\delta_i$ exist. We choose the smallest $\delta_i$ such that $\delta=\min \delta_i$. Then we can assume that the $D_i(z_i)$ are small disjoint disks with radius $\delta$.  We choose the polynomial $P(z)=\prod (z-z_i)^{d_i}$.
Let  $f\in C^\infty_0(\mathbb B_5\backslash \cup_{i} D_i(z_i))$ be a real-valued function.

Based on the above preparations, we are ready to show the proof of Theorem \ref{th1}.
\begin{proof}[Proof of Theorem \ref{th1}]
As discussed above, we first consider the case $d_i\geq 1$.
We choose a cut off function $\psi\in C^\infty_{0}(\mathbb B_1\backslash \cup_{i}  D_i(z_i))$ with the following properties: \\
(1) $\psi(z)=1$ if $|z|\leq \frac{1}{2}$ and $|z-z_i|\geq 2 \delta$, \medskip \\
(2) $|\nabla \psi|\leq C$ and $|\triangle \psi|\leq C$  if $|z|\geq \frac{1}{2}$, \medskip \\
(3) $|\nabla \psi|\leq C \delta^{-1}$ and $|\triangle \psi|\leq C \delta^{-2}$  if $|z-z_i|\leq 2 \delta$. \\

We substitute $f=\psi u$ into the Carleman estimates (\ref{main-Car}).
Direct calculations show that
\begin{align}
\triangle f+ \tilde{b}(x)\cdot \nabla f+ \tilde{c}(x) f&=\triangle \psi u+2\nabla \psi\cdot \nabla u+\psi \triangle u+ \tilde{b}(x)\cdot\nabla \psi u+\tilde{b}(x)\cdot\nabla u\psi+\tilde{c}(x)u\psi \nonumber \\
&=\triangle \psi u+2\nabla \psi\cdot \nabla u+ \tilde{b}(x)\cdot\nabla \psi u,
\end{align}
where we have used the equation (\ref{target-1}). In the neighborhood $|z-z_i|\leq 2\delta$, by the vanishing order of $u$ at $z_i$, we can check that
\begin{align}
|\nabla \psi u|\leq C \delta^{d_i},  \quad |\triangle\psi  u|\leq C \delta^{d_i}, \  \ \mbox{and} \ |\nabla u\cdot \nabla \psi |\leq C \delta^{d_i}.
\end{align}
Near the neighborhood $|z-z_i|\leq 2\delta$, it holds that
 \begin{align}
 |P|^{-2}|\triangle (\psi u)+\tilde{b}(x)\cdot \nabla (\psi u)+\tilde{c}(x) \psi u|^2 \leq C \delta^{-2d_i}  \delta^{2d_i}\leq C.
 \end{align}
 Thus, $|P|^{-2}|\triangle (\psi u)+\tilde{b}(x)\cdot \nabla (\psi u)+\tilde{c}(x) u|^2$ is uniformly integrable  near the singular points $z_i$ as $\delta\to 0$.
If $|z|\geq \frac{1}{2}$, from the assumption of $\psi$, we can see that
\begin{align*}
|\triangle f+\tilde{b}\cdot \nabla f+\tilde{c}(x) f|&=|\triangle \psi u+2\nabla \psi\cdot \nabla u+ \tilde{b}\cdot\nabla \psi u| \nonumber \\
&\leq C(|u|+|\nabla u|).
\end{align*}
Substituting  $f=\psi u$ in the Carleman estimates (\ref{main-Car}) and applying the Lebegue dominated convergence theorem as $\delta\to 0$, we have
\begin{align*}
\int_{\frac{1}{2}\leq |z|\leq 1} (|u|^2+|\nabla u|^2) |P|^{-2} e^{CN|z|^2} \geq CN^2 \int_{ |z|\leq \frac{1}{3}} |u|^2 |P|^{-2} e^{CN|z|^2}.
\end{align*}
We take the maximum and minimum of $|P|$ out of the integrations. It holds that
\begin{align}
 e^{CN} \max_{\frac{1}{2}\leq |z|\leq 1} |P|^{-2} \int_{\frac{1}{2}\leq |z|\leq 1} (|u|^2+|\nabla u|^2) \geq CN^2 \min_{ |z|\leq \frac{1}{3}}  |P|^{-2}  \int_{ |z|\leq \frac{1}{3}} |u|^2.
\end{align}
By standard elliptic estimates, we get that
\begin{align}
 e^{CN} \max_{\frac{1}{2}\leq |z|\leq 1} |P|^{-2} \int_{\frac{2}{5} \leq |z|\leq \frac{6}{5}} |u|^2 \geq CN^2 \min_{ |z|\leq \frac{1}{3}}  |P|^{-2}  \int_{ |z|\leq \frac{1}{3}} |u|^2.
 \label{compare}
\end{align}
We claim that
\begin{align}
e^{C\sum d_i}&\leq \frac{ \min_{ |z|\leq \frac{1}{3}}  |P|^{-2}}{ \max_{\frac{1}{2}\leq |z|\leq 1} |P|^{-2}  }.
\label{poly-com}
\end{align}

To show (\ref{poly-com}), it is equivalent to prove
\begin{align}
e^{C\sum d_i}&\leq \big( \frac{ \min_{ \frac{1}{2}\leq |z|\leq 1 }  |P|}{ \max_{ |z|\leq \frac{1}{3}} |P| }\big)^2.
\label{poly-com2}
\end{align}
Since $|z_i|\leq \frac{1}{5}$, we have
\begin{align}
\min_{ \frac{1}{2}\leq |z|\leq 1 }  |P|\geq (\frac{1}{2}-\frac{1}{5})^{C\sum d_i}=(\frac{3}{10})^{C\sum d_i}
\label{poly1}
\end{align}
and
\begin{align}
\max_{ |z|\leq \frac{1}{3} }  |P|\leq (\frac{1}{3}-\frac{1}{5})^{C\sum d_i}=(\frac{2}{15})^{C\sum d_i}.
\label{poly2}
\end{align}
Together with (\ref{poly1}) and (\ref{poly2}), we arrive at (\ref{poly-com2}), i.e. (\ref{poly-com}).
It follows from (\ref{compare}) and (\ref{poly-com}) that
\begin{align}
e^{C\sum d_i}\leq \frac{ e^{CN} \int_{\frac{2}{5}\leq |z|\leq \frac{6}{5}} |u|^2} {\int_{|z|\leq \frac{1}{3}} |u|^2  }.
\end{align}
By the almost monotonicity of the double index  $N(u, r)$ in (\ref{mono-sin-1}), it holds that
\begin{align}
\frac{\int_{\frac{2}{5}\leq |z|\leq\frac{6}{5}} |u|^2} {\int_{|z|\leq  \frac{1}{3}} |u|^2}\leq e^{CN}.
\end{align}
Thus, we have
\begin{align}
\sum d_i\leq CN.
\label{con-n}
\end{align}
Hence,  we arrive at the conclusion in the theorem in the case $d_i\geq 1$.

Now we treat the case for singular points with vanishing order two, i.e. $d_i=0$. We consider the polynomial $P_1(z)=\prod (z-z_i)^\frac{1}{2}$.
%(we can identify $z_i$ as $(x_i, y_i)$ in $\mathbb R^2$).
We want to replace $P(z)$ in the above arguments by $P_1(z)$. Near the singular point $z_i$, we can still show that $ |P_1(z)|^{-2}|\triangle (\psi u)+\tilde{b}(x)\cdot \nabla (\psi u)+\tilde{c}(x) \psi u|^2$ is uniformly integrable as $\delta\to 0$. If $|z-z_i|\leq 2\delta$, we can check
\begin{align*}
|\nabla \psi u|\leq C ,  \quad |\triangle\psi  u|\leq C , \  \ \mbox{and} \ |\nabla u\cdot \nabla \psi |\leq C .
\end{align*}
Thus, \begin{align} |P_1(z)|^{-2}|\triangle (\psi u)+\tilde{b}(x)\cdot \nabla (\psi u)+\tilde{c}(x) \psi u|^2 \leq C\delta^{-1},\end{align}
which is uniformly integral in $\mathbb B_5$. However, $P_1(z)$ is not defined as single-valued holomorphic function. As indicated in \cite{DF90}, we can pass to a finite branched cover of the disc $\mathbb{D}_5$ punctured at $z_i$. Since the Carleman estimates  (\ref{complex-car}) are obtained by integration by parts, these Carleman estimates are arrived in a straightforward manner. The integrand in these estimates involves function such as $f$ and $|P_1|$ which are independent of the sheet. Therefore, we still have the Carleman estimates (\ref{main-Car}) in the punctured disc. Following the arguments as we did to get $(\ref{con-n})$ for $n_i\geq 2$, the conclusion $\sum\limits_{z_i\in \mathbb B_{1/5}} 1\leq CN $ will still be arrived for $n_i=1$. Recall that we have done a rescaling argument to have $R_0=5$. Thus, the estimate (\ref{con-n}) implies that
\begin{align}
H^{0}(\{\mathcal{S}\cap \mathbb B_{\frac{R_0}{25}}\}   )\leq C  N.
\end{align}
It follows from (\ref{dou-com}) that $N\leq C\mathcal{N}(\frac{3R_0}{2})$ for some large $N$. Therefore, the proof of the theorem is arrived.

\end{proof}

% Another approach: We consider the branched out to connect $z_i$ in $\mathbb D_5$. Let us say the branched out be the set $\mathcal{L}$. The set $\mathcal{L}$ has Euclidean measure zero.  We can still do the integration by parts on the  $\mathbb D_5\backslash \{\cap \mathbb D_i(z_i)\cap \mathcal{L}\}$. The integration on the boundary (i.e. on the two sides of  $\mathcal{L}$) from the integration by parts will be opposite. Thus there are no terms for the integration on the boundary. Hence Carleman estimates (\ref{main-Car}) can still be arrived.

Next we show the upper bound of critical sets for (\ref{main-equ-c}).  As before, we may choose $R_0\leq \frac{\hat{r}_0}{2}$ and set $R_0=5$ and $\hat{r}_0=10$ by rescaling. Suppose that the critical points $|z_i|\leq \frac{1}{5}$.
 Near the critical point $z_i$,  $u(z)$ can be approximated by the function $P_c + u(z_i)$ in $\hat{D}_i(z_i)$, where $P_c$ is some homogeneous polynomial with degree $d_i+2$ and $\hat{D}_i(z_i)$ is some small disk with radius $\delta_i$.  Furthermore, $\nabla u(z)$ can be approximated by $\nabla P_c$ in $\hat{D}_i(z_i)$, see  e.g. \cite{HW53}. In particular, if $u(z_i)=0$, then $z_i$ is the singular point.  Since critical points are finitely many discrete points, we choose the smallest $\delta_i$ such that $\delta=\min \delta_i$ and assume that the $\mathbb D_i(z_i)$ are small disjoint disks with radius $\delta$.

\begin{proof}[Proof of Theorem \ref{th2}]
 We still consider the case $d_i\geq 1$ at the beginning. If $f\in C^\infty_0(\mathbb D_5\backslash \cup_{i}\mathbb D_i(z_i))$ and $P(z)=\prod (z-z_i)^{d_i}$, by choosing $\alpha=\hat{N}$ and $F=f$ in (\ref{complex-car}),  it follows that the following Carleman estimates hold
\begin{align}
\int_{\mathbb D_5}|\overline{\partial} f  |^2 |P|^{-2} e^{C \hat{N}|z|^2}\geq C \hat{N} \int_{\mathbb D_5}| f|^2 |P|^{-2} e^{C \hat{N}|z|^2}.
\label{Carle-cri}
\end{align}

We choose the same real-valued cut-off function $\psi$ in the last theorem. That is, the cut-off function $\psi\in C^\infty_{0} (\mathbb D_1\backslash \cup_{i} \mathbb D_i(z_i))$ satisfies the following properties: \\
(1) $\psi(z)=1$ if $|z|\leq \frac{1}{2}$ and $|z-z_i|\geq 2 \delta$, \medskip \\
(2) $|\nabla \psi|\leq C$ and $|\triangle \psi|\leq C$  if $|z|\geq \frac{1}{2}$, \medskip \\
(3) $|\nabla \psi|\leq C \delta^{-1}$ and $|\triangle \psi|\leq C \delta^{-2}$  if $|z-z_i|\leq 2 \delta$. \\

 Substituting $f=\psi \partial u$ in the Carleman estimates (\ref{Carle-cri}), we have
\begin{align}
\int_{\mathbb D_5}|\overline{\partial}\psi \partial u  |^2 |P|^{-2} e^{C \hat{N}|z|^2}+
\int_{\mathbb D_5}\psi^2 | \triangle u |^2|P|^{-2} e^{C \hat{N}|z|^2}
\geq C \hat{N} \int_{\mathbb D_5}\psi^2| \partial u|^2|P|^{-2} e^{C \hat{N}|z|^2}.
\end{align}
From the equation (\ref{target-2}) and the fact that $\tilde{b}$ is bounded, we get
\begin{align*}
\int_{\mathbb D_5}|\overline{\partial}\psi \partial u  |^2 |P|^{-2} e^{C \hat{N}|z|^2} \geq C \hat{N} \int_{\mathbb D_5}\psi^2| \partial u|^2 |P|^{-2} e^{C \hat{N}|z|^2}.
\end{align*}
Thus, we have
\begin{align}
\int_{\mathbb B_5}|\nabla \psi|^2|\nabla u|^2 |P|^{-2} e^{C \hat{N}|z|^2} \geq C \hat{N} \int_{\mathbb B_3}\psi^2| \nabla u|^2 |P|^{-2} e^{C \hat{N}|z|^2}.
\label{carle-real}
\end{align}

Near the critical point $z_i$, for $|z-z_i|\leq 2 \delta$, we can check that
\begin{align*}
|\nabla \psi \cdot \nabla u|\leq C\delta^{d_i}. \quad
\end{align*}
Thus,
\begin{align*}
|\nabla \psi|^2|\nabla u|^2 |P|^{-2}\leq C\delta^{2d_i} \delta^{-2 d_i}\leq C.
\end{align*}
Then $|\nabla \psi|^2|\nabla u|^2 |P|^{-2}$ is uniformly integrated as $\delta\to 0$.
From the assumption of $\psi$, applying the dominated convergence theorem as $\delta\to 0$, we have
\begin{align}
\int_{\frac{1}{2}\leq |z|\leq 1} |\nabla u|^2 |P|^{-2} e^{C\hat{N}|z|^2} \geq C\hat{N} \int_{ |z|\leq \frac{1}{3}} |\nabla u|^2 |P|^{-2} e^{C\hat{N}|z|^2}.
\end{align}
By taking the maximum and minimum of $|P|$, we get
\begin{align}
 e^{C\hat{N}} \max_{\frac{1}{2}\leq |z|\leq 1} |P|^{-2} \int_{\frac{1}{2}\leq |z|\leq 1} |\nabla u|^2 \geq C\hat{N} \min_{ |z|\leq \frac{1}{3}}  |P|^{-2}  \int_{ |z|\leq \frac{1}{3}} |\nabla u|^2.
 \label{cri-sum}
 \end{align}
From (\ref{poly-com}), we have
\begin{align}
e^{C\sum d_i}&\leq \frac{ \min_{ |z|\leq \frac{1}{3}}  |P|^{-2}}{ \max_{\frac{1}{2}\leq |z|\leq 1} |P|^{-2}  }.
\label{poly-com-2}
\end{align}
It follows from (\ref{cri-sum}) and (\ref{poly-com-2}) that
\begin{align}
e^{C\sum d_i}\leq \frac{ e^{C\hat{N}} \int_{\frac{1}{2}\leq |z|\leq 1} |\nabla u|^2} {\int_{|z|\leq \frac{1}{3}} |\nabla u|^2  }.
\end{align}
By the almost monotonicity of the double index of $\hat{N}(\nabla u, r)$ in (\ref{mono-cri-1}), we show that
\begin{align}
\frac{\int_{\frac{1}{2}\leq |z|\leq 1} |\nabla u|^2} {\int_{|z|\leq \frac{1}{3}} |\nabla u|^2}\leq e^{C\hat{N}}.
\end{align}
Hence, we arrive at
\begin{align}
\sum d_i\leq C\hat{N}.
\label{con-n-2}
\end{align}
Therefore, in the case $d_i\geq 1$, the conclusion of the theorem follows.

Now we deal with the case for critical points with vanishing order two, i.e. $d_i=0$. We follow the arguments in the last theorem for singular sets with vanishing order two.  We replace $P(z)$ in the above arguments by $P_1(z)$, where $P_1(z)=\prod (z-z_i)^\frac{1}{2}$.  If $|z-z_i|\leq 2\delta$, we can check
$ |\nabla u\cdot \nabla \psi |\leq C$.
Thus, \begin{align} |P_1(z)|^{-2}|\nabla u\cdot \nabla \psi |^2 \leq C\delta^{-1},\end{align}
which is uniformly integral in $\mathbb B_5$ as $\delta\to 0$. However, $P_1(z)$ is not defined as single-valued holomorphic function. We can pass to a finite branched cover of the disc $\mathbb{D}_5$ punctured at $z_i$. Since the Carleman estimates  (\ref{Carle-cri}) are obtained by integration by parts, these Carleman estimates are arrived similarly. The integrand in these estimates involves functions such as $f$ and $|P_1|$ which are independent of the sheet. Therefore, we still have the Carleman estimates (\ref{carle-real}) in the punctured disc. Following the arguments  as we did to get $(\ref{con-n-2})$ for $ d_i \geq 1$, the conclusion $\sum\limits_{z_i\in \mathbb B_{1/5}} 1\leq C\hat{N} $ will  be arrived for $d_i=0$.
Rescaling back to $R_0$, the estimate (\ref{con-n-2}) yields that
\begin{align}
H^{0}(\{\mathcal{C}\cap \mathbb B_{\frac{R_0}{25}}\}   )\leq C  \hat{N}.
\end{align}
It follows from (\ref{dou-com-c}) that $\hat{N}\leq C\hat{\mathcal{N}}(\frac{3R_0}{2})$ for some large $\hat{N}$.
This completes the proof of the theorem.
\end{proof}

The rest of the section is devoted to the discussion of upper bound of critical points with large coefficients in the equations in the plane. To study the local growth of gradient near each point, we introduce
$$\mathcal{\hat{N}}(x, r)=\frac{ r\int_{\mathbb B_r(x)} |\nabla u|^2} {\int_{\partial \mathbb B_r(x)} ( u-u(x))^2}, $$
where $\mathbb B_r(x)$ is the ball centered at $x$ with radius $r$. It is known that
\begin{align}
\mathcal{\hat{N}}(x, r)\leq C \mathcal{\hat{N}}(2R_0)
\label{shift-mono}
\end{align}
for $x\in \mathbb B_{\frac{R_0}{4}}$ and $0<r\leq \frac{3R_0}{2}$, see e.g. \cite{NV17}.
The arguments in Theorem \ref{th2} can be applied to study the upper bound of critical points for elliptic equations with a large drift term. We consider the elliptic equations
\begin{align}
{\rm div}( A(x)\nabla u)+ \lambda b(x)\cdot \nabla u =0 \quad \quad \mbox{in}  \ \mathbb B_5,
\label{drift}
\end{align}
where $A(x)=(a_{ij}(x))_{2\times 2}$ satisfies the assumptions (\ref{ellip}) and (\ref{lipschitz}), $b(x)$ satisfies the condition (\ref{bboud}), and possibly $\lambda \to \infty$. Let $\tilde{\mathcal{N}}=\hat{\mathcal{N}}(2R_0)$. If $\lambda\leq C\tilde{\mathcal{N}}$, we can perform the same argument in Theorem \ref{th2} directly. Thus, we will obtain the upper bound
\begin{align}
H^{0}(\{\mathcal{C}\cap \mathbb B_\frac{R_0}{25}\}   )\leq C\tilde{\mathcal{N}}.
\end{align}
 If $\lambda\geq C\tilde{\mathcal{N}}$, we first do some rescaling arguments. Let $v(x)=u(\frac{\tilde{\mathcal{N}}}{\lambda}x+x_0)$ for $x_0\in \mathbb B_{\frac{R_0}{4}}$. We consider the critical sets of $u$ in $\mathbb B_{\frac{5\tilde{\mathcal{N}}}{\lambda}}(x_0)$. Thus, $v(x)$ satisfies the equation
 \begin{align*}
 {\rm div}( \bar {A}(x)\nabla v)+ \tilde{\mathcal{N}} \bar{b}(x)\cdot \nabla v=0 \quad \quad \mbox{in}  \ \mathbb B_5,
 \end{align*}
where $\bar {A}(x)=(\bar{a}_{ij}(x))_{2\times 2}=(a_{ij}(\frac{\tilde{\mathcal{N}}}{\lambda}x))_{2\times 2}$ and $\bar{b}(x)=b(\frac{\tilde{\mathcal{N}}}{\lambda}x)$.
By the arguments in the proof of Theorem \ref{th2}, we can show that
\begin{align}
H^{0}(\{\mathbb B_\frac{R_0}{25}|\ |\nabla v(x)|=0 \}   )\leq C \hat{N}(\nabla v, \frac{R_0}{2}).
\end{align}
From (\ref{shift-mono}) and (\ref{dou-com-c}), it holds that
\begin{align*}
\hat{N}(\nabla v, \frac{R_0}{2})\leq \hat{\mathcal{N}}(x_0, \frac{3R_0 \tilde{\mathcal{N}}}{2\lambda})\leq C \tilde{\mathcal{N}}.
\end{align*}
Thus, we have
\begin{align}
H^{0}(\{\mathbb B_\frac{\tilde{\mathcal{N}}R_0}{25\lambda}(x_0) |\ |\nabla u(x)|=0 \}   )\leq C\tilde{\mathcal{N}}.
\end{align}
Covering the ball $\mathbb B_\frac{R_0}{25}$ with $C\frac{\lambda^2}{\tilde{\mathcal{N}}^2}$ number of $\mathbb B_{\frac{\tilde{\mathcal{N}}R_0}{25\lambda}}(x_0) $ balls for $x_0\in \mathbb B_\frac{R_0}{25}$, we obtain that
\begin{align}
H^{0}(\{\mathbb B_\frac{R_0}{25}|\ |\nabla u(x)|=0 \}   )\leq C\frac{\lambda^2} {\tilde{\mathcal{N}}}.
\end{align}
For the upper bound of singular points with a large first order term or zero order term, see \cite{DF90}, \cite{Z16}.

%%\bibliographystyle{amsplain}
%%\bibliography{mybib}
\end{document}